\documentclass[a4paper,12pt]{amsart}
\usepackage{fullpage, verbatim}
\usepackage{enumerate, amsfonts, amssymb, amsthm, mathrsfs}
\usepackage{xcolor}

\usepackage[all]{xy}
\usepackage[T1]{fontenc}

\usepackage[colorlinks=true]{hyperref}
\numberwithin{equation}{section}
\usepackage{cleveref}

\newcommand\CA{{\mathscr A}}

\newcommand\FD{\mathfrak D}
\newcommand\FDinfty{\FD(\CA,-\infty)}

\newcommand\BBC{{\mathbb C}}

\newcommand\BBN{{\mathbb N}}

\newcommand\BBZ{{\mathbb Z}}

\newcommand\Sym{\operatorname{Sym}}

\newcommand\SalphaH{{S_{\langle{\alpha_H}\rangle}}}

\DeclareMathOperator{\spec}{Spec}

\DeclareMathOperator{\diag}{diag}

\newcommand\Hodgefil{\mathcal{H}}
\newcommand\Hodgedec{\mathcal{G}}

\newcommand\Binf{B_\infty}

\newcommand{\Mprime}{M^{(0)}}

\newcommand\coexp{\operatorname{coexp}}

\newcommand\Der{{\operatorname{Der}}}
\newcommand\DerS{{\Der_S}}
\newcommand\DerF{{\Der_F}}
\newcommand\DerSW{{\Der_S^W}}

\newcommand\DerR{{\Der_R}}

\newcommand\InvSpecial{F^{\mathrm{fl}}}
\newcommand\hh{\Delta}
\newcommand\VanLoc{\mathcal{H}}

\newcommand\Fix{{\operatorname{Fix}}}

\newcommand\GL{\operatorname{GL}}

\newcommand\pdeg{\operatorname{pdeg}}

\newcommand\refl{\mathcal{R}}

\newcommand{\prim}{\nabla_{\! D}}
\newcommand{\primk}[1][k]{\prim^{#1}}

\newcommand{\priminv}{\prim^{-1}}

\newcommand{\eqdef}{:=}

\newcommand{\pmat}[1]{\begin{pmatrix} #1 \end{pmatrix}}

\newcommand{\colvect}{\pmat{\partial_{t_1} \\ \vdots \\ \partial_{t_\ell}}}
\newcommand{\colvecx}{\pmat{\partial_{x_1} \\ \vdots \\ \partial_{x_\ell}}}
\newcommand{\colvecEta}{\pmat{\eta_\ell \\ \vdots \\ \eta_1}}

\newcommand{\rowvect}{(\partial_{t_1}, \ldots, \partial_{t_\ell})}

\newcommand{\rowvecEta}{(\eta_\ell, \ldots, \eta_1)}

\newcommand{\Jtx}{J_{\partial\bbt/\partial\bbx}}
\newcommand{\Jxt}{J_{\partial\bbx/\partial\bbt}}

\newcommand{\bbt}{{\textbf{t}}}

\newcommand{\bbx}{{\textbf{x}}}

\newcommand{\tr}{\mathrm{tr}}

\definecolor{darkblue}{rgb}{0,0,0.7} 
\newcommand{\darkblue}{\color{darkblue}} 
\newcommand{\defn}[1]{\emph{\darkblue #1}} 

\newcommand{\ie}{{i.e.}} 
\newcommand{\cf}{\text{cf.}} 

\newcommand{\one}{{1\!\!1}} 

\usepackage[colorinlistoftodos]{todonotes}
\theoremstyle{plain}
\newtheorem{lemma}[equation]{Lemma}
\newtheorem{theorem}[equation]{Theorem}

\newtheorem{corollary}[equation]{Corollary}
\newtheorem{proposition}[equation]{Proposition}
\theoremstyle{definition}
\newtheorem{definition}[equation]{Definition}
\newtheorem{remark}[equation]{Remark}


\begin{document}

\title[Logarithmic vector fields on reflection arrangements]
{A Hodge filtration of logarithmic vector fields \\ for well-generated complex reflection groups}

\author[T.~Abe]{Takuro Abe}
\address[T.~Abe]{Department of Mathematics, Rikkyo University, Tokyo, Japan}
\email{abetaku@rikkyo.ac.jp}

\author[G.~R\"ohrle]{Gerhard R\"ohrle}
\address[G.~R\"ohrle]{Fakult\"at f\"ur Mathematik, Ruhr-Universit\"at Bochum, Germany}
\email{gerhard.roehrle@rub.de}

\author[C.~Stump]{Christian Stump}
\address[C.~Stump]{Fakult\"at f\"ur Mathematik, Ruhr-Universit\"at Bochum, Germany}
\email{christian.stump@rub.de}

\author[M.~Yoshinaga]{Masahiko Yoshinaga}
\address[M.~Yoshinaga]{Department of Mathematics, Hokkaido University, Sapporo, Japan}
\email{yoshinaga@math.sci.hokudai.ac.jp}


\allowdisplaybreaks

\begin{abstract}
  Given an irreducible well-generated complex reflection group, we construct an explicit basis for the module of vector fields with logarithmic poles along its reflection arrangement.
  This construction yields in particular a Hodge filtration of that module.
  Our approach is based on a detailed analysis of a flat connection applied to the primitive vector field.
  This generalizes and unifies analogous results for real reflection groups.
\end{abstract}

\maketitle



\section{Introduction}


The study of vector fields with logarithmic poles along the reflection arrangement of a finite Coxeter group inside a real vector space has been a particularly active and fruitful area of recent research.
Most importantly, Abe-Terao use Saito's primitive derivation to construct an explicit basis of this vector field in~\cite{AT2010}, thereby extending Saito's Hodge filtration of invariant polynomial derivations to a Hodge filtration of the module of invariant vector fields with logarithmic poles.
Also, Wakamiko identified the concept of a universal vector field as a crucial ingredient in his construction of an explicit basis in~\cite[Sec.~2]{Wak2011}.

Based on a recent extension of Saito's primitive derivation and Hodge filtration to well-generated unitary reflection groups in~\cite{HMRS2017}, we establish analogues of the above constructions in this more general setting.
Specifically, we provide a framework to extend~\cite[Thms.~1.1 \& 1.2]{AT2010} to well-generated unitary reflection groups (\Cref{thm:Hodgefiltration} and \Cref{cor:Hodgefiltration}), and derive universality results generalizing~\cite[Sec.~2]{Wak2011} (\Cref{thm:universality,,thm:zetapmshifts,,thm:shifteduniversality}).
Because of the new explicit form of the flat connection in~\cite{HMRS2017}, the approach we provide here has not appeared in the literature even in the real case, simplifying several arguments.

\medskip

We give a brief comment on the terminology of a ``Hodge filtration'' in our context. This was introduced
by K.~Saito in \cite{Sai1983}. The primitive form is a special element in the
relative de Rham cohomology group for the deformation of an isolated hypersurface 
singularity, which satisfies lots of nice properties.
For simple singularities of type $ADE$, the parameter space
of the deformation is identified
with the quotient of the Cartan subalgebra by the Weyl group.
For a given logarithmic vector field, we can differentiate (Gauss-Manin connection)
the primitive form by the vector field, which gives an identification between the module
of logarithmic vector fields and the relative de Rham cohomology group.
The primitive derivation defines a natural filtration on which the primitive
derivation acts by degree shift by one. This structue is very similar to
the behavior of Hodge filtrations on the cohomology of the fiber with respect to
the Gauss-Manin connection of a proper smooth holomorphic map
(the so-called Griffith transversality). This is the reason why the filtration
on the module of logarithmic vector fields is called the Hodge filtration. 

\medskip

The paper is organized as follows.
\Cref{sec:preliminaries} contains preliminaries about complex hyperplane arrangements with $\BBZ$-valued multiplicities, and recalls the needed background on unitary reflection groups and the necessary properties of flat connections for well-generated unitary reflection groups.
In \Cref{sec:generalarrangements} we provide an analogue of Saito's criterion for $\BBZ$-valued multiplicities, introduce universal vector fields and derive several of their properties.
\Cref{sec:hodgefil} contains the main results of this paper as presented above.

For the benefit of the reader, in our proofs we have separated the properties that follow directly from the existence of a universal vector field from those that only hold in the case of well-generated complex reflection groups (where we do have such a universal vector field) see \Cref{sec:universality1,,sec:universality2}.

\medskip

Recall that a complex reflection group $W$ is called well-generated provided it is generated by $\ell$ reflections, where $\ell$ is the dimension of its reflection 
representation. 
See Section \ref{sec:prelimreflectiongroupps} for an equivalent definition used throughout the paper; here the exponents and coexponents of $W$ add up to the Coxeter number of $W$ in a particular fashion, see~\eqref{eq:wellgen}.

\medskip

We finish this introduction with a brief discussion of the crucial differences and similarities of the situations for real and complex reflection arrangements;  for definitions we refer to the sections below.
Denote by $\nu : \CA \rightarrow \BBZ$ a general $\BBZ$-valued multiplicity function on the reflection arrangement~$\CA$ of a well-generated irreducible unitary reflection group~$W$, and by $\omega : \CA \rightarrow \BBZ$ the multiplicity function assigning to a reflecting hyperplane $H$ the order of its pointwise stabilizer, $\omega(H) = e_H= |W_H|$.

\begin{enumerate}
  \item If $W$ is real, then $\omega \equiv 2$.
        In particular, for the module of derivations $\FD(\CA,\nu-1)$ of $\CA$ with multiplicity $\nu-1$, we have
        \[
          \FD(\CA,\nu-1) = \FD(\CA,\nu-\omega+1).
        \]
        It turns out that the counterpart of the module of derivations $\FD(\CA,\nu-1)$ that is used in the literature in the real case is the module $\FD(\CA,\nu-\omega+1)$ in the general complex case.
        One crucial instance, where a $(-1)$-multiplicity in the real case is replaced by a $(-\omega+1)$-multiplicity in the complex case, is demonstrated in \Cref{thm:zetapmshifts}, which generalizes the analogous result from the real case in~\cite[Thm.~2.7]{Wak2011}.
  \item It is well known that in the real case there is an isomorphism of graded modules 
        \[
          \FD(\CA,-1) = \FD(\CA,-\omega+1) \cong \Omega(\CA,1),
        \]
        where $\Omega(\CA,1)$ is the module of differential $1$-forms. As explained in \Cref{rem:tocheck}, by degree comparison, this isomorphism does not extend to complex reflection arrangements that are not the complexification of a real arrangement.

  \item The explicit form of the flat connection that was exhibited in~\cite{HMRS2017}, see \Cref{prop:mainingredient}, is the crucial ingredient in the construction of the bases in \Cref{thm:Hodgefiltration} and of the Hodge filtration in \Cref{cor:Hodgefiltration}.
  Using this explicit form and a \emph{system of flat invariants} and of \emph{flat derivations} has the additional benefit of simplifying the arguments for real reflection arrangements.
  We remark that, while the existence of flat derivations has been known for some time, see~\cite{Bes2015}, the existence of flat invariants for well-generated complex reflection groups was only discovered quite recently in~\cite{KMS2015}.

  \item This paper does not deal with equivariant multiplicities (multiplicities that are constant along hyperplane orbits) as provided in the real case in~\cite{ATW2012}.
  That construction is based on a case-by-case analysis of primitive vector fields along reflection subgroups generated by orbits of reflections.
  We hope for a general conceptual approach to such equivariant multiplicities.
\end{enumerate}


\section{Preliminaries}
\label{sec:preliminaries}

\subsection{Multi-arrangements of hyperplanes and their derivations}

Let~$V$ be a finite-dimensional complex vector space of dimension~$\ell$ and fix a Hermitian form~$I : V \times V \rightarrow \BBC$ on~$V$.
Let $S = \Sym(V^*)$ denote the ring of \defn{polynomial functions} on~$V$ and let~$F$ denote its quotient field of \defn{rational functions}.
If $x_1, \ldots, x_\ell$ is a basis of~$V^*$, we identify~$S$ with the polynomial ring $\BBC[x_1, \ldots , x_\ell]$ and~$F$ with the field of rational functions $\BBC(x_1,\ldots,x_\ell)$.
Letting $S_p$ denote the $\BBC$-subspace of~$S$ consisting of the homogeneous polynomials of degree $p$ (along with $0$),~$S$ is naturally $\BBZ$-graded by $S = \oplus_{p \in \BBZ}S_p$, where we set $S_p = 0$ for $p < 0$.
Denote by $\DerS$ the \defn{$S$-module of $\BBC$-derivations} of~$S$, and by $\DerF$ the \defn{$F$-module of $\BBC$-derivations} of~$F$.
Then $\partial_{x_1}, \ldots, \partial_{x_\ell}$ is an $S$-basis of $\DerS$ and an~$F$-basis of $\DerF$.
We say that $\theta \in \DerF$ is \defn{homogeneous of polynomial degree $p-q$} provided $\theta = \sum \frac{f_i}{g_i} \partial_{x_i}$, where $f_i \in S_p$ and $g_i \in S_q\setminus\{0\}$ for each $1 \le i \le \ell$.
In this case we write $\pdeg \theta = p-q$.
Recall that, for $\theta \in \DerF$, we have $\theta = \theta(x_1)\partial_{x_1} + \dots + \theta(x_\ell)\partial_{x_\ell}$.
The \defn{Saito matrix} of $\theta_1, \ldots, \theta_\ell \in \DerF$ is given by
\[
  M(\theta_1, \ldots, \theta_\ell) \eqdef
  \begin{bmatrix}
    \theta_1(x_1) & \cdots & \theta_1(x_\ell) \\
    \vdots & \ddots & \vdots \\
    \theta_\ell(x_1) & \cdots & \theta_\ell(x_\ell)
  \end{bmatrix},
\]
see~\cite[Def.~4.11]{orlikterao:arrangements}.
That is, the Saito matrix collects in the $i$-th row the coefficients of $\theta_i$ in the $F$-basis $\{ \partial_{x_1},\ldots,\partial_{x_\ell}\}$ of $\DerF$.

\medskip

Extending Ziegler's concept of an $\BBN$-valued multiplicity function from~\cite{ziegler:multiarrangements}, a \defn{multi-arrangement} $(\CA,\nu)$ is an arrangement~$\CA$ together with a \defn{multiplicity function} $\nu : \CA \to \BBZ$ assigning to each hyperplane~$H\in\CA$ a multiplicity $\nu(H) \in \BBZ$.
We use the term multi-arrangement even though we allow hyperplanes to have negative multiplicities, while disregarding hyperplanes of multiplicity zero.
We only consider \emph{central} multi-arrangements $(\CA,\nu)$, \ie, $0 \in H$ for every $H \in \CA$.
In this case, we fix $\alpha_H \in V^*$ with $H = \ker(\alpha_H)$ for $H \in \CA$ which we scale so that $I^*(\alpha_H,\alpha_H) = 1$.
The \defn{defining rational function} $Q(\CA,\nu)$ is
\[
  Q(\CA,\nu) \eqdef \prod_{H \in \CA} \alpha_H^{\nu(H)}
\]
which we sometimes abbreviate as $Q_\nu \eqdef Q(\CA,\nu)$.
We set $|\nu| \eqdef \sum_{H \in \CA} \nu(H)$ to be the degree of $Q_\nu$ and separate $Q_\nu$ via $Q_\nu = Q_+/Q_-$ with
\begin{equation}
\label{eq:qplus}
  Q_+ \eqdef \prod_{\substack{H \in \CA\\ \nu(H) > 0}} \alpha_H^{ \nu(H)}, \qquad
  Q_- \eqdef \prod_{\substack{H \in \CA\\ \nu(H) < 0}} \alpha_H^{-\nu(H)}.
\end{equation}

To deal with general multiplicity functions on the arrangement~$\CA$, we also set
\[
  Q_\CA \eqdef  Q(\CA, 1) = \prod_{H \in \CA} \alpha_H
\]
to be the defining polynomial of the \defn{simple arrangement}~$\CA$ which we later abbreviate as~$Q = Q_\CA$ when~$\CA$ is clear from the context.

\begin{definition}
\label{def:logarithmicvectorfield}
  Let $\SalphaH$ be the localization of the ring~$S$ at the prime ideal $\langle{\alpha_H}\rangle$.
  Setting $S' \eqdef S[Q_\CA^{-1}]$, the $S$-module of \defn{logarithmic vector fields} on $\CA$ is defined by
  \begin{equation*}
    \FDinfty \eqdef \big\{ \theta \in \Der_{S'} \mid \theta(\beta) \in \SalphaH \text{ for } H \in \CA \text{ and } \beta \in V^* \text{ with } I^*(\alpha_H,\beta) = 0 \big\}.
  \end{equation*}
  Given a multiplicity function $\nu : \CA \rightarrow \BBZ$, define the $S$-module of \defn{$(\CA, \nu)$-derivations} by
  \begin{equation*}
    \FD(\CA,\nu) \eqdef \big\{ \theta \in \FDinfty \mid \theta(\alpha_H) \in \alpha_H^{\nu(H)}\SalphaH \text{ for } H \in \CA\big\}.
  \end{equation*}
\end{definition}

Observe that in general this definition depends on the chosen Hermitian form, while $\FD(\CA,0) = \DerS$.
We remark also that the given definition is an equivalent reformulation of the definitions given in~\cite{AT2010,Wak2011}.

\medskip

We record the following basic containment property for multiplicities $\mu \geq \nu$, i.e., $\mu(H) \geq \nu(H)$ for all $H \in \CA$.

\begin{lemma}
\label{lem:compmultiplicities}
  We have $\mu > \nu$ if and only if $\FD(\CA,\mu) \subsetneqq \FD(\CA,\nu)$.
  In particular, we have $\FD(\CA,\mu) \subseteq \DerS$ if and only if $\mu \geq 0$.
\end{lemma}

\begin{proof}
  The reverse implication of the asserted equivalence, as well as the fact that $\mu \geq \nu$ implies $\FD(\CA,\mu) \subseteq \FD(\CA,\nu)$ are immediate from the definition.
  For $\mu > \nu$, let $H \in \CA$ with $\mu(H) > \nu(H)$.
  If $\nu(H) \geq 0$, then
  $Q_+\partial_{\alpha_H} \in \FD(\CA,\nu) \setminus \FD(\CA,\mu)$,
  and if $\nu(H) < 0$, then $\alpha_H^{\nu(H)} Q_+\partial_{\alpha_H}  \in \FD(\CA,\nu) \setminus \FD(\CA,\mu)$, where $Q_+$ is as in~\eqref{eq:qplus}.
\end{proof}

The multi-arrangement $(\CA, \nu)$ is \defn{free} if $\FD(\CA, \nu)$ is a free~$S$-module.
In this case, $\FD(\CA, \nu)$ admits a basis $\{\theta_1, \ldots, \theta_\ell\}$ of homogeneous derivations, see~\cite[Thm~A.20]{orlikterao:arrangements}.
While the $\theta_i$'s are not unique, their polynomial degrees $\pdeg \theta_i$ are, see~\cite[Prop.~A.24]{orlikterao:arrangements}.
The multiset of these polynomial degrees is the multiset of \defn{exponents} of the free multi-arrangement $(\CA,\nu)$.
It is denoted by
\[
  \exp(\CA,\nu) \eqdef \big\{ \pdeg(\theta_1),\ldots, \pdeg(\theta_\ell) \big\}.
\]

\subsection{Unitary reflection groups and their reflection arrangements}
\label{sec:prelimreflectiongroupps}

Let~$W$ be an irreducible unitary reflection group with reflection representation~$V \cong \BBC^\ell$, and let~$I$ be the associated~$W$-invariant Hermitian inner form.
We refer to~\cite{HMRS2017} and the references therein for all necessary background material on reflection groups.
Denote the set of reflections of~$W$ by $\refl = \refl(W)$, and the associated reflection arrangement in~$V$ by $\CA = \CA(W)$.
For $H \in \CA$, let $e_H$ denote the order of the pointwise stabilizer of~$H$ in~$W$.
The \defn{Coxeter number} of~$W$ is given by
\begin{equation}
\label{eq:coxeter}
  h = h_W \eqdef \frac{1}{\ell} \sum_{H \in \CA} e_H = \frac{1}{\ell}\Big(|\refl| + |\CA|\Big),
\end{equation}
generalizing the usual Coxeter number of a real reflection group to irreducible unitary reflection groups.

\medskip

Before proceeding, we record, without proof, the following well-known property of the $W$-action on polynomial functions.

\begin{lemma}
\label{lem:quasiinvarinantpoly}
  Let $g \in S$ and let $H \in \CA$ with corresponding reflection~$s \in \refl$ with $\Fix(s) = H$ of order~$e_H$.
  Let~$\epsilon$ be a primitive $e_H$-th root of unity such that $s(\alpha_H) = \epsilon\alpha_H$.
  For $1 \leq k < e_H$, we have
  \[
    s(g) = \epsilon^k g \ \Longrightarrow\ g \in \alpha_H^k \cdot S.
  \]
\end{lemma}

Results of Shephard and Todd~\cite{shephardtodd} and of Chevalley~\cite{chevalley} distinguish unitary reflection groups as those finite subgroups of $\GL(V)$ for which the invariant subalgebra of the action on the symmetric algebra $S = \Sym(V^*)\cong\BBC[x_1,\ldots,x_\ell] $ yields again a polynomial algebra,
\[
  S^W \cong \BBC[f_1,\ldots,f_\ell]
\]
for homogeneous polynomials $f_1,\dots,f_\ell$ of degrees $d_i = \deg f_i$ with $d_1 \leq \dots \leq d_\ell$.
Let $\exp(W) = \{ e_1 \leq \dots \leq e_\ell \}$ be the \defn{exponents} of~$W$, where $e_i = d_i - 1$ 
and let $\coexp(W) = \{ e_1^* \leq \dots \leq e_\ell^* \}$ be the \defn{coexponents} of~$W$, \cf~\cite[Def.~6.50]{orlikterao:arrangements}.
The group~$W$ is called \defn{well-generated} if it is generated by~$\ell$ reflections, or, equivalently, if
\begin{equation}
	\label{eq:wellgen}
  e_i + e_{\ell+1-i}^* = h \quad\text{for } 1 \leq i \leq \ell,
\end{equation}
see~\cite{OS1980}.
Terao showed in~\cite{terao:freereflections} that the reflection arrangement $\CA = \CA(W)$ of~$W$ is free,
and that the exponents coincide with the coexponents of~$W$, \cf~\cite[Thm.~6.60]{orlikterao:arrangements},
\begin{equation}
\label{eq:expcoexp}
  \exp \CA = \coexp(W).
\end{equation}
This shows that $e_1^* = 1$ (coming from the \emph{Euler derivation}, see below), implying that $e_\ell = h-1$ in the case of well-generated groups.

\medskip

For later usage, we recall from~\cite[Sec.~3]{HMRS2017} the diagonal matrix
\begin{equation}
\label{eq:Binf}
  \Binf \eqdef \tfrac{1}{h}\diag(e_i) - \one_\ell.
\end{equation}

\medskip

Next we recall from~\cite{HMRS2017} the \defn{order multiplicity} $\omega$ of the reflection arrangement
$\CA = \CA(W)$ defined by $\omega(H) = e_H$ for $H \in \CA$.
In other words, the multiplicities are chosen so that the defining polynomial $Q(\CA, \omega)$ of the multi-arrangement $(\CA, \omega)$ is the \defn{discriminant} of~$W$, \cf~\cite[Def.~6.44]{orlikterao:arrangements},
\[
  Q(\CA, \omega) = \prod_{H \in \CA(W)} \alpha_H^{e_H} = QJ,
\]
where $Q = Q_\CA = \prod_{H \in \CA} \alpha_H$ as before, and $J = J_\CA \eqdef \prod_{H \in \CA} \alpha_H^{e(H)-1}$.

\subsection{Flat connections for well-generated unitary reflection groups}
\label{sec:prelimflat}

Throughout this subsection let~$W$ be a well-generated unitary reflection group.
Let $\nabla : \DerF \times \DerF \rightarrow \DerF$ be the connection defined by 
\begin{equation}
  \label{eq:flatconnect} \nabla_\theta(\phi) = \sum_i \theta (p_i) \partial_{x_i}
\end{equation}
for $\theta, \phi \in \DerF$ with $\phi = \sum p_i \partial_{x_i}$, or equivalently, an affine connection which has $\partial_{x_i}$ as a flat section, i.e., $\nabla_\theta(\partial_{x_i}) = 0$.
Recall that $\nabla$ is $F$-linear in the first parameter and $\BBC$-linear in the second, satisfying the Leibniz rule
\[
  \nabla_\theta(p \phi) = \theta(p)\phi + p \nabla_\theta(\phi)
\]
for $\theta,\phi \in \DerF$ and $p \in F$. 
Alternatively, this can be characterized by
\begin{align}
  \nabla_\theta(\phi)(\alpha) = \theta(\phi(\alpha)) \label{eq:defnabla}
\end{align}
for all $\alpha \in V^*$.
Observe that for
~$\theta,\phi$  homogeneous,~\eqref{eq:flatconnect} 
implies that the derivation $\nabla_\theta(\phi)$ is again homogeneous with  
polynomial degree
\begin{equation}
  \label{eq:pdegnabla}\pdeg\big(\nabla_\theta(\phi)\big) = \pdeg(\theta) + \pdeg(\phi) - 1.
\end{equation}

\bigskip

Let $\InvSpecial_1,\ldots,\InvSpecial_\ell$ be the special homogeneous fundamental invariants in $\BBC[\bbx]$ with $\bbx = (x_1,\ldots,x_\ell)$, as given in~\cite[Thm.~5.5]{KMS2015}.
Recall that $\deg\big(\InvSpecial_i\big) = d_i$ and
\begin{equation*}
  \BBC[\InvSpecial_1,\ldots,\InvSpecial_\ell] \cong S^W.
\end{equation*}
Consider indeterminates $\bbt = (t_1,\ldots,t_\ell)$ together with the map $t_i \mapsto \InvSpecial_i$ giving an isomorphism 
\[
  R \eqdef \BBC[\bbt] \cong \BBC[\InvSpecial_1,\ldots,\InvSpecial_\ell].
\]
By slight abuse of notation, here
and in the rest of the paper, the variable $t_i$ and its image $\InvSpecial_i$ under this isomorphism are identified.
Set moreover $T \eqdef \BBC[\bbt'] = \BBC[t_1,\ldots,t_{\ell-1}]$, the subring of~$\BBC[\bbt]$ generated by $\bbt' = (t_1,\ldots,t_{\ell-1})$.

\medskip

As usual, set
\begin{equation*}
  \Jtx \eqdef
    \colvecx (t_1,\ldots,t_\ell) = 
    \begin{bmatrix}
      \partial t_1 / \partial x_1 & \cdots & \partial t_\ell / \partial x_1 \\
      \vdots & \ddots & \vdots \\
      \partial t_1 / \partial x_\ell & \cdots & \partial t_\ell / \partial x_\ell
    \end{bmatrix}  \in \BBC[\bbx]^{\ell\times\ell}
\end{equation*}
with inverse matrix $\Jxt \eqdef \Jtx^{-1} = \rowvect^\tr (x_1,\ldots,x_\ell)$.
It is well-known that $\det \Jtx \doteq J = \prod_{H \in \CA} \alpha_H^{e_H-1}$, see~\cite[Thm.~6.42]{orlikterao:arrangements}.
Here and elsewhere the symbol~$\doteq$ denotes, as usual, equality up to a non-zero complex constant factor.

\medskip

The \defn{primitive vector field} $D \eqdef \partial_{t_\ell} \in \DerR$ is given by 
\begin{equation}
\label{eq:defD}
  D = \det \Jxt
    \begin{vmatrix}
      \frac{\partial t_1}{\partial x_1} & \cdots & \frac{\partial t_{\ell-1}}{\partial x_1} & \frac{\partial}{\partial x_1} \\
      \vdots & \ddots & \vdots & \vdots \\
      \frac{\partial t_1}{\partial x_\ell} & \cdots & \frac{\partial t_{\ell-1}}{\partial x_\ell} & \frac{\partial}{\partial x_\ell} \\
    \end{vmatrix}
    \in \DerF.
\end{equation}
In particular,~$D$ is homogeneous of degree $\pdeg(D) = - e_\ell = -(h-1)$, where we observe that $h = d_\ell > d_{\ell-1}$.
The primitive vector field~$D$ is thus, up to a non-zero complex constant, independent of the given choice of fundamental invariants.

\bigskip

Consider $X \eqdef V \big/ W = \spec(\BBC[\bbt])$ and let $\hh(\bbt) \in R$ be the discriminant of~$W$ given by 
\[
  \hh\big(\InvSpecial_1(\bbx),\ldots,\InvSpecial_\ell(\bbx)\big) = \prod_{H \in \CA} \alpha_H^{e_H}
\]
with vanishing locus $\VanLoc \eqdef \{ \overline p \in X \mid \hh(\overline p) = 0 \}$, \cf~\cite[Def.~6.44]{orlikterao:arrangements}.
Let~$\DerR$ be the $R$-module of logarithmic vector fields, and let
\[
  \DerR(-\log\hh) \eqdef \big\{ \theta \in \DerR \mid \theta\hh \in R\hh \big\}
\]
be the module of logarithmic vector fields along~$\VanLoc$.
We have an $R$-isomorphism between such logarithmic vector fields and~$W$-invariant $S$-derivations,
\begin{equation}
\label{eq:DerRDerSW}
  \DerR(-\log\hh) \cong \DerSW,
\end{equation}
and $\DerR(-\log\hh)$ is a free $R$-module, \cf~\cite[Cor.~6.58]{orlikterao:arrangements}.

\medskip

Bessis showed in~\cite[Thm.~2.4]{Bes2015} that there exists a \defn{system of flat homogeneous derivations} $\{ \eta_1,\ldots,\eta_\ell \}$ of $\DerR(-\log\hh)$ with $\pdeg\eta_i = e_i^*$ being the $i$-th coexponent of~$W$.
This means, its Saito matrix
\begin{equation}
\label{eq:flatderivations}
  M_\eta \eqdef M(\eta_\ell, \ldots, \eta_1) =
  \begin{bmatrix}
    \eta_\ell(t_1) & \cdots & \eta_\ell(t_\ell) \\
    \vdots & \ddots & \vdots \\
    \eta_1(t_1) & \cdots & \eta_1(t_\ell)
  \end{bmatrix}
\end{equation}
decomposes as
\begin{align}
\label{eq:MV}
  M_\eta = t_\ell \one_\ell + \Mprime(\bbt')
\end{align}
with $\Mprime(\bbt') \in \BBC[\bbt']^{\ell\times\ell}$.
As before, we have $\rowvecEta^\tr = M_\eta \rowvect^\tr$.
Moreover, we obtain that $\hh(\bbt)$ is a monic polynomial in~$t_\ell$ with coefficients in $T$, \ie,
\[
  \hh(\bbt) = t_\ell^\ell + a_{\ell-1}(\bbt')t_\ell^{\ell-1} + \ldots + a_{1}(\bbt')t_\ell + a_{0}(\bbt'),
\]
where $ a_{i}(\bbt') \in T$ for $1 \le i \le \ell$.
As observed in~\cite[Lem.~3.8]{KMS2015}, such a system of flat homogeneous derivations is unique, and we have that $\eta_1 = \tfrac{1}{h} \sum d_i t_i \partial_{t_i} \in \DerR$ is the \defn{Euler vector field} mapped to the (scaled) \defn{Euler derivation}
\begin{equation}
\label{eq:Euler}
  E \eqdef \tfrac{1}{h} \sum x_i \partial_{x_i} \in \DerSW
\end{equation}
under the isomorphism in~\eqref{eq:DerRDerSW}, see~\cite[Lem.~3.5]{KMS2015}.

\medskip

We recall from~\cite[Prop.~3.15]{HMRS2017} the following proposition which is the key ingredient in the present considerations.
It involves the diagonal matrix $\Binf$ given in~\eqref{eq:Binf}.

\begin{proposition}
\label{prop:mainingredient}
  We have $T$-isomorphisms
  \begin{equation*}
    \begin{aligned}
    \prim &: \DerR(-\log\hh) \longrightarrow \DerR \\
    \priminv &: \DerR \longrightarrow \DerR(-\log\hh)
    \end{aligned}\label{eq:nablainv}
  \end{equation*}
  given by
  \begin{equation*}
    \begin{aligned}
      \prim\colvect    &= -M_\eta^{-1}(\Binf + \one_\ell) \colvect \\
      \priminv\colvect &= -B^{-1}_\infty \colvecEta = -B^{-1}_\infty M_\eta \colvect.
    \end{aligned}
  \end{equation*}
\end{proposition}

From this explicit description of~$\prim$, the following Hodge filtration of~$\Der_R$ was deduced in ~\cite[Prop.~3.15]{HMRS2017}, generalizing Saito's construction for Coxeter arrangements in~\cite{Sai1993} to the situation for well-generated unitary reflection groups.
Let $\Hodgedec_0$ be the $T$-submodule of $\Der_R$ generated by $\partial_{t_1},\ldots,\partial_{t_\ell}$ and let
\begin{equation}
\label{eq:hodgedec}
  \Hodgedec_k \eqdef \primk[-k](\Hodgedec_0) \text{ for } k \in \BBZ.
\end{equation}
Then the Hodge filtration of $\Der_R$ is given by
\begin{equation}
\label{eq:hodgefil}
  \Hodgefil^{(k)}_0 \eqdef \bigoplus_{i \geq k} \Hodgedec_i.
\end{equation}

\section{A general version of Saito's criterion and universal vector fields}
\label{sec:generalarrangements}

Throughout this section, we consider a general multi-arrangement  $(\CA,\nu)$ in $V \cong \BBC^\ell$ with multiplicity function $\nu : \CA \to \BBZ$.

\subsection{A general version of Saito's criterion}

For later usage we generalize Ziegler's version of Saito's criterion~\cite[Thm.~8]{ziegler:multiarrangements} for multi-arrangements from $\BBN$- to $\BBZ$-valued multiplicities.

\begin{theorem}
\label{thm:zieglersaito}
  For a multiplicity function~$\nu : \CA \rightarrow \BBZ$, let $\theta_1, \ldots, \theta_\ell \in \FD(\CA, \nu)$.
  Then the following are equivalent:
  \begin{enumerate}[(i)]
  \item\label{eq:zieglersaito1} $\{\theta_1, \ldots, \theta_\ell\}$ is an ~$S$-basis of $\FD(\CA, \nu)$.
  \item\label{eq:zieglersaito2} $\det \big(M(\theta_1, \ldots, \theta_\ell)\big) \doteq  Q(\CA, \nu)$.
  \end{enumerate}
  Moreover, if each $\theta_i$ is homogeneous, then (i) and (ii) are equivalent 
  to
  \begin{enumerate}[(i)]
    \setcounter{enumi}{2}
    \item\label{eq:zieglersaito3} $\theta_1, \ldots, \theta_\ell$ are linearly independent over~$S$ and $\sum \pdeg \theta_i = |\nu|$.
  \end{enumerate}
\end{theorem}

We require two lemmas for the proof of \Cref{thm:zieglersaito}.
Recall from \eqref{eq:qplus} that we write $Q_\nu = Q_+/Q_-$ with $Q_+, Q_-\in S$
for the defining rational function $Q_\nu =Q(\CA, \nu)$.

\begin{lemma}
\label{lem:nushift}
  Let $\theta \in \FD(\CA,\nu)$.
  Then $Q_-\theta \in \DerS$.
\end{lemma}

\begin{proof}
  Fix $H \in \CA$ and let $\alpha_H = x_1,x_2,\ldots,x_\ell$ be an orthonormal basis of~$V^*$.
  Set $Q \eqdef Q_\CA$ and $Q_H \eqdef Q / \alpha_H$.
  Write $\theta = \sum_i (f_i / Q^n) \partial_{x_i}$ with $f_i \in S$ and $n \geq 0$.
  Then $\theta \in \FD(\CA,\nu)$ means that
  \[
    f_1 \in \alpha_H^{n+\nu(H)} \cdot S \qquad
    f_i \in \alpha_H^{n} \cdot S \text{ for } 2 \leq i \leq \ell,
  \]
  or, equivalently, that $\theta = \sum_i (f'_i / Q_H^n) \partial_{x_i}$ with
  \begin{equation}
  \label{eq:nushift}
    f'_1 \in \alpha_H^{\nu(H)} \cdot S \qquad
    f'_i \in S \text{ for } 2 \leq i \leq \ell.
  \end{equation}
  With $Q_H^{-n} \in \SalphaH$, we obtain that $Q_- f_i \in \SalphaH$ for all $1 \leq i \leq \ell$ and thus $\theta \in \Der_{\SalphaH}$.
  Since this holds for any $H \in \CA$, we conclude that $Q_-\theta \in \DerS$.
\end{proof}

In the subsequent arguments, we use the following abbreviation.
Given $\theta_1,\ldots,\theta_\ell \in \DerF$, we set
\[
  \theta_1 \wedge \dots \wedge \theta_\ell \eqdef \det\big(M(\theta_1,\ldots,\theta_\ell)\big) = \det\big( \theta_i(x_j) \big)_{1 \leq i,j \leq \ell} \in F.
\]

\begin{lemma}
\label{lem:zieglersaito}
  Let $\nu : \CA \rightarrow \BBZ$ and let $\theta_1,\ldots,\theta_\ell \in \FD(\CA,\nu)$.
  Then
  \[
    \theta_1 \wedge \dots \wedge \theta_\ell \in Q_\nu \cdot S.
  \]
\end{lemma}

\begin{proof}
  Let $H \in \CA$ and consider the orthonormal basis $\alpha_H = x_1,x_2,\ldots,x_\ell$ of~$V^*$.
  Since $\theta_i \in \FD(\CA,\nu)$, we have
  \[
    \theta_i(\alpha_H) \in \alpha_H^{\nu(H)} \cdot \SalphaH \quad\text{and}\quad \theta_i(x_j) \in \SalphaH \text{ for } 2 \leq j \leq \ell.
  \]
  This implies that $\theta_1 \wedge \dots \wedge \theta_\ell \in \alpha_H^{\nu(H)}\cdot \SalphaH$.
  As this holds for any $H \in \CA$, the lemma follows.
\end{proof}

\begin{proof}[Proof of \Cref{thm:zieglersaito}]
  First we show that~\eqref{eq:zieglersaito2} implies~\eqref{eq:zieglersaito1}.
  As $\theta_1 \wedge \dots \wedge \theta_\ell \neq 0$, we readily see that $\theta_1,\ldots,\theta_\ell$ are $S$-independent. So we aim to show that they span $\FD(\CA,\nu)$.
  To this end, let $\theta \in \FD(\CA,\nu)$.
  Since $\theta_i \in \FD(\CA,\nu)$ for $1\leq i \leq \ell$, we have $Q_-\theta_i \in \DerS$ by \Cref{lem:nushift}.
  Let~$M := M(Q_-\theta_1,\ldots,Q_-\theta_\ell)$ and observe that
  \[
    \det(M) = Q_-\theta_1 \wedge \dots \wedge Q_-\theta_\ell \doteq Q_-^\ell Q_\nu.
  \]
  Set $N = M^{-1}$ and $\tilde N = \det(M) N \in S^{\ell\times\ell}$.
  Since
  \[
    \det(M)\partial_{x_i} = \sum_j \tilde N_{ij}Q_-\theta_j,
  \]
  we see that $Q_-\theta_1,\ldots,Q_-\theta_\ell$ generate $Q_-^\ell Q_\nu \cdot \DerS$ over~$S$.
  With $Q_-\theta \in \DerS$, we obtain
  \[
    \det(M)Q_-\theta = Q_+Q_-^\ell \theta = \sum_{i=1}^\ell f_i Q_- \theta_i
  \]
  for $f_1,\ldots,f_\ell \in S$, or, equivalently, $Q_+Q_-^{\ell-1} \theta = \sum f_i \theta_i$.
  It thus remains to show that each $f_i$ is divisible by $Q_+Q_-^{\ell-1} = Q_\nu Q_-^\ell$.
  We have
  \begin{align*}
    f_i Q_+ &\doteq f_i Q_- ( \theta_1 \wedge \dots \wedge \theta_\ell ) \\
            &= \theta_1 \wedge \dots \wedge \theta_{i-1} \wedge f_i Q_- \theta_i \wedge \theta_{i+1} \wedge \dots \wedge \theta_\ell \\
            &= Q_+ Q_-^\ell  (\theta_1 \wedge \dots \wedge \theta_{i-1} \wedge \theta \wedge \theta_{i+1} \wedge \dots \wedge \theta_\ell) \\
            &= Q_+ Q_-^\ell g_i Q_\nu,
  \end{align*}
  for some $g_i \in S$ given by \Cref{lem:zieglersaito}.
  This implies $f_i = Q_\nu Q_-^\ell g_i$, as desired.

  \medskip

  We next show~\eqref{eq:zieglersaito1} implies~\eqref{eq:zieglersaito2}.
  By \Cref{lem:zieglersaito}, we may write $\theta_1\wedge \dots\wedge \theta_\ell = f Q_\nu = fQ_+/Q_-$ for some $f \in S$.
  Since $\theta_1,\ldots,\theta_\ell$ are $S$-independent, we have~$f \neq 0$.
  We thus need to show that~$f$ is constant.
  Now, let $H \in \CA$ and again assume that $\alpha_H = x_1,\ldots,x_\ell$ are orthonormal coordinates.
  Set
  \[
    Q_+ = \alpha_H^{m_+} Q'_+ \quad\text{and}\quad Q_- = \alpha_H^{m_-} Q'_-
  \]
  such that $Q'_+$ and $Q'_-$ are both not divisible by $\alpha_H$.
  
  First, observe that
  \[
    Q_+\partial_{\alpha_H}, Q'_+\partial_{x_2},\ldots,Q'_+\partial_{x_\ell} \in \FD(\CA,\nu),
  \]
  by definition.
  This implies
  \begin{align*}
    Q_+(Q'_+)^{\ell-1} &= Q_+\partial_{\alpha_H} \wedge Q'_+\partial_{x_2} \wedge\ldots\wedge Q'_+\partial_{x_\ell} \\
                       &= \det(g_{ij}) \big(\theta_1 \wedge \dots \wedge \theta_\ell\big) \\
                       &= gf Q_+/Q_-
  \end{align*}
  where $(g_{ij})$ is the matrix with entries in~$S$ expressing $Q_+\partial_{\alpha_H}, Q'_+\partial_{x_2},\ldots,Q'_+\partial_{x_\ell}$ in terms of the basis $\theta_1,\dots,\theta_\ell$, and where we wrote $g = \det(g_{ij}) \in S$.
  This means that~$f$ is a factor of $Q_-(Q'_+)^{\ell-1}$ for the given~$H \in \CA$.
  Repeating this argument for every hyperplane in~$\CA$, we obtain that~$f$ is a factor of $Q_-$.

  Second, we also observe that
  \[
    \frac{Q_+}{\alpha_H^{m_-}}\partial_{\alpha_H}, Q_+\partial_{x_2},\ldots,Q_+\partial_{x_\ell} \in \FD(\CA,\nu)
  \]
  and analogously deduce that
  \[
    \frac{Q_+^\ell}{\alpha_H^{m_-}} = g'f Q_+/Q_-
  \]
  for some $g' \in S$.
  This now means that~$f$ is a factor of $Q'_-(Q_+)^{\ell-1}$ for the given~$H \in \CA$.
  Once again repeating this for every hyperplane in~$\CA$, we obtain that~$f$ is a factor of $(Q_+)^{\ell-1}$.
  As $Q_-$ and $Q_+$ do not have any common non-scalar factors, we thus deduce that~$f$ is constant.

  The argument for~\eqref{eq:zieglersaito3} is standard.
  If $\{\theta_1,\ldots,\theta_\ell\}$ forms a homogeneous $S$-basis of $\FD(\CA,\nu)$, then~\eqref{eq:zieglersaito3} is obviously satisfied.
  Vice versa,~\eqref{eq:zieglersaito3} implies that the determinant of $M(\theta_1,\ldots,\theta_\ell)$ equals $fQ_\nu$ for some $f \in S \setminus \{0\}$ by independence, and that $\deg(f) = 0$, because the degrees of $\det \big(M(\theta_1,\ldots,\theta_\ell)\big)$ and of $Q_\nu$ coincide.
\end{proof}

\subsection{Properties of universal vector fields on multi-arrangements}
\label{sec:universality1}

In this section, we define universal vector fields for general arrangements and arbitrary multiplicity functions and explore properties of logarithmic vector fields that follow from the existence of universal vector fields.
Throughout, we have separated the arguments as detailed as possible to clarify their exact interplay.
In \Cref{sec:universality2}, we exhibit natural occurrences and properties of universal vector fields in the case of reflection arrangements.
We start with the definition of universality depending on a multiplicity $\nu : \CA \rightarrow \BBZ$ which is a slight generalization of the $\mathbf{k}$-universality in the real case in~\cite[Def.~2.2]{Wak2011}.

\medskip

Recall the definition of the affine connection~$\nabla$ in~\eqref{eq:flatconnect} having $\partial_{x_i}$ as a flat section.
\begin{definition}
\label{def:universal}
  Let $(\CA,\nu)$ be a multi-arrangement for a multiplicity function $\nu : \CA \rightarrow \BBZ$ and let $\zeta \in \FDinfty$ be homogeneous.
  Then~$\zeta$ is called \defn{$\nu$-universal} provided the map
  \[
    \Phi_\zeta :
      \DerF \longrightarrow \FD(\CA,\nu); \quad
      \theta \longmapsto \nabla_\theta(\zeta)
  \]
  is an isomorphism of $S$-modules.
\end{definition}

Observe that $\Phi_\zeta$ is $S$-linear by definition.
The $\nu$-universality of~$\zeta$ thus means that $\Phi_\zeta : \DerS \longrightarrow \FD(\CA,\nu)$ is well-defined and bijective.
In particular, this implies for an $\nu$-universal~$\zeta$ that $(\CA,\nu)$ is free and
\[
  \big\{\nabla_{\partial_{x_i}}(\zeta) \mid 1 \leq i \leq \ell \big\}
\]
is an $S$-module basis of $\FD(\CA,\nu)$.

\medskip

We begin with the following observation from~\cite[Ex.~2.3]{Wak2011}.

\begin{lemma}
  The Euler derivation $E$ defined in~\eqref{eq:Euler} is $\mathbf{0}$-universal.
\end{lemma}
\begin{proof}
  Since $\Phi_E(\delta) = \nabla_\delta(E) = \delta$ for any $\delta \in \DerS$, the statement follows.
\end{proof}

\begin{lemma}
\label{lem:nablaoninf}
  Let $\theta \in \DerS$ and $\zeta \in \FDinfty$.
  Then $\nabla_\theta(\zeta) \in \FDinfty$.
\end{lemma}

\begin{proof}
  Let $H \in \CA$ and let $\beta \in V^*$ such that $I^*(\alpha_H,\beta) = 0$.
  Since $\zeta \in \FDinfty$, we have $\zeta(\beta) = f/Q_H$ for $f \in S$ and $\widetilde Q_H = (\frac{Q}{\alpha_H})^n$ for some $n \geq 0$.
  Using \eqref{eq:defnabla}, we obtain
  \[
    \nabla_\theta(\zeta)(\beta)
      = \theta(f/\widetilde Q_H)
      = \frac{1}{\widetilde Q_H^2}(\theta(f)\widetilde Q_H - f\theta(\widetilde Q_H)).
  \]
  Since $\theta \in \DerS$ by assumption, we obtain $\nabla_\theta(\zeta)(\beta) \in \SalphaH$.
  This holds for any $H \in \CA$ and we deduce the statement.
\end{proof}

\begin{lemma}
\label{lem:booleanmultiplicitywelldefined}
  Let $\mu : \CA \rightarrow \{0,1\}$ and let $\zeta \in \FDinfty$.
  Then $\zeta \in \FD(\CA,\nu+1)$ if and only if $\Phi_\zeta : \FD(\CA,\mu) \longrightarrow \FD(\CA,\nu+\mu)$ is well-defined.
\end{lemma}

\begin{proof}
  Let $\zeta \in \FD(\CA,\nu+1)$.
  We aim to show that $\nabla_\theta(\zeta) \in \FD(\CA,\nu+\mu)$ for all $\theta \in \FD(\CA,\mu)$.
  By \Cref{lem:nablaoninf} we have that $\nabla_\theta(\zeta) \in \FDinfty$ in this case.
  The following argument is analogous to the proof of the previous lemma.
  Let $H \in \CA$.
  Since $\zeta \in \FD(\CA,\nu+1)$, we have that
  \[
    \zeta(\alpha_H) = \alpha_H^{\nu(H)+1}f/Q_H
  \]
  for $f \in S$ and $Q_H = (\frac{Q}{\alpha_H})^n$ for some $n \geq 0$.
  Using \eqref{eq:defnabla}, for $\theta \in \FD(\CA,\mu)$ we have 
  \begin{equation}
    \begin{aligned}
      \nabla_\theta(\zeta)(\alpha_H)
        &= \theta\big(\alpha_H^{\nu(H)+1}f/Q_H\big) \\
        &= \frac{1}{Q_H^2}\big( \theta(\alpha_H^{\nu(H)+1})fQ_H + \alpha_H^{\nu(H)+1}\theta(f)Q_H - \alpha_H^{\nu(H)+1}f\theta(Q_H) \big) \\[5pt]
        &= \frac{\alpha_H^{\nu(H)}}{Q_H^2}\big( (\nu(H)+1)\theta(\alpha_H)fQ_H + \alpha_H\theta(f)Q_H - \alpha_H f\theta(Q_H) \big).
    \end{aligned}
    \label{eq:derivationoffraction1}
  \end{equation}
  This yields $\nabla_\theta(\zeta)(\alpha_H) \in \alpha_H^{\nu(H)}\SalphaH$.
  If $\mu(H) = 1$, we moreover get $\theta(\alpha_H) \in \alpha_H \cdot S$ by the defining property of $\theta \in \FD(\CA,\mu)$.
  In this case, we thus get an additional factor $\alpha_H$ in \eqref{eq:derivationoffraction1}
  and obtain $\nabla_\theta(\zeta)(\alpha_H) \in \alpha_H^{\nu(H)+1}\SalphaH$. 
  In any event $\Phi_\zeta(\theta)$ belongs to $\FD(\CA,\nu+\mu)$.

  \medskip

  Assume now that $\Phi_\zeta : \FD(\CA,\mu) \longrightarrow \FD(\CA,\nu+\mu)$ is well-defined.
    Let~$\tau$ be the maximal multiplicity function such that $\zeta \in \FD(\CA,\tau+1)$.
  This is,
  \[
    \zeta(\alpha_H) = \alpha_H^{\tau(H)+1}f/Q_H
  \]
  for $f \in S$ and $Q_H = (\frac{Q}{\alpha_H})^n$ for some $n \geq 0$, such that $f \notin \alpha_H\cdot S$.
  As in~\eqref{eq:derivationoffraction1}, for $\theta \in \FD(\CA,\mu)$ we obtain
  \begin{equation}
  \label{eq:derivationoffraction}
    \nabla_\theta(\zeta)(\alpha_H) = \frac{\alpha_H^{\tau(H)}}{Q_H^2}\big( (\tau(H)+1)\theta(\alpha_H)fQ_H + \alpha_H\theta(f)Q_H - \alpha_H f\theta(Q_H) \big).
  \end{equation}
  Observe that $\theta = Q_\mu\partial_{\alpha_H} \in \FD(\CA,\mu)$.
  If $\mu(H) = 0$, then for $\theta = Q_\mu\partial_{\alpha_H}$, we have $\nabla_\theta(\zeta)(\alpha_H) \in \alpha_H^{\tau(H)} \SalphaH$ but $\nabla_\theta(\zeta)(\alpha_H) \notin \alpha_H^{\tau(H)+1} \SalphaH$, since the first summand in~\eqref{eq:derivationoffraction} is not divisible by $\alpha_H$ while the other two are.
  By the well-definedness of $\Phi_\zeta$, we thus have $\tau(H) \geq \nu(H)$.

  Similarly, if $\mu(H) = 1$, then the first summand~\eqref{eq:derivationoffraction} is divisible by $\alpha_H$ but not by $\alpha_H^2$, while the other two are divisible by $\alpha_H^2$.
  Again by the well-definedness of $\Phi_\zeta$, we have that $\tau(H) \geq \nu(H)$.
  We conclude that $\tau \geq \nu$ and so by definition, we get
  \[
    \zeta \in \FD(\CA,\tau+1) \subseteq \FD(\CA,\nu+1). \qedhere
  \]
\end{proof}

The following lemma is the analogue of~\cite[Prop.~2.6(3)]{Wak2011} in our setting.

\begin{lemma}
\label{lem:universalunit}
  Let~$\zeta$ be $\nu$-universal.
  Then~$\zeta \notin \FD(\CA,\mu+1)$ for any $\mu > \nu$.
\end{lemma}

\begin{proof}
  Let  $\mu > \nu$ and suppose that $\zeta \in \FD(\CA,\mu+1)$.
  \Cref{lem:booleanmultiplicitywelldefined,,lem:compmultiplicities} then imply that $\Phi_\zeta(\DerS) \subseteq \FD(\CA,\mu) \subsetneqq \FD(\CA,\nu)$.
  This contradicts the fact that $\Phi_\zeta(\DerS) = \FD(\CA,\nu)$ owed to the $\nu$-universality of~$\zeta$.
\end{proof}

\Cref{lem:booleanmultiplicitywelldefined,,lem:universalunit} give the following theorem, generalizing~\cite[Thm.~2]{AY2009} to the complex case and to the more general notion of universality above, see also~\cite[Prop.~2.7]{Wak2011}.

\begin{theorem}
\label{thm:01multi}
  Let $\mu : \CA \rightarrow \{0,1\}$ and let~$\zeta$ be $\nu$-universal.
  Then
  \[
    \Phi_\zeta : \FD(\CA,\mu) \longrightarrow \FD(\CA,\nu+\mu)
  \]
  is an isomorphism of $S$-modules.
\end{theorem}

Observe that in the theorem we do not require any freeness assumption on $(\CA,\mu)$.
Together with \Cref{thm:universality} below, it thus generalizes~\cite[Thm.~3.22]{HMRS2017}.

\begin{proof}
  We already know that $\Phi_\zeta$ is $S$-linear and well-defined by \Cref{lem:booleanmultiplicitywelldefined}.
  Moreover, the universality of~$\zeta$ implies that $\Phi_\zeta$ is injective on~$\DerS$ and thus on $\FD(\CA,\mu) \subseteq \DerS$.
  It thus remains to show that $\Phi_\zeta$ is surjective.
  To this end, let $\phi \in \FD(\CA,\nu+\mu)$.
  Since $\FD(\CA,\nu+\mu) \subseteq \FD(\CA,\nu)$, we may write $\phi = \nabla_\theta(\zeta)$ for some $\theta \in \DerS$.
  We aim to show that $\theta \in \FD(\CA,\mu)$.
  As in the proof of \Cref{lem:booleanmultiplicitywelldefined}, we have
  \[
    \zeta(\alpha_H) = \alpha_H^{\nu(H)+1}f/Q_H
  \]
  for some $f \in S$ and $Q_H = (\frac{Q}{\alpha_H})^n$ with $n \geq 0$. Since $\phi = \nabla_\theta(\zeta) \in \FD(\CA,\nu+\mu)$, we have
  \begin{equation}
    \begin{aligned}
 \phi & = \nabla_\theta(\zeta) \\
 	& = \frac{\alpha_H^{\nu(H)}}{Q_H^2}\Big[ (\nu(H)+1)\theta(\alpha_H)fQ_H + \alpha_H \big( \theta(f)Q_H - f\theta(Q_H)\big) \Big] \in \alpha_H^{\nu(H)+\mu(H)}\cdot \SalphaH.
    \end{aligned}
    \label{eq:derivationoffraction2}
  \end{equation}
  Observe that \Cref{lem:universalunit} implies that~$f \notin \alpha_H \cdot S$ by the $\nu$-universality of~$\zeta$.

  If $\mu(H) = 0$, there is no condition from \eqref{eq:derivationoffraction2} on $\theta$, and so trivially  
  $\theta(\alpha_H) \in \alpha_H^0 \cdot S = S$, and if $\mu(H) = 1$, then $\theta(\alpha_H)\cdot f \in \alpha_H\cdot S$ implying that $\theta(\alpha_H) \in \alpha_H \cdot S$, because~$f$ is not divisible by~$\alpha_H$, by~\eqref{eq:derivationoffraction2}.
  Consequently, $\theta \in \FD(\CA,\mu)$, as desired.
\end{proof}

\section{A Hodge filtration and universality for reflection arrangements}
\label{sec:hodgefil}

For the remainder of the paper, we fix $W$ to be a well-generated irreducible complex reflection group with reflection arrangement~$\CA = \CA(W)$, and the order multiplicity $\omega : \CA \rightarrow \BBZ$ given by $\omega(H) = e_H$.
Recall from \eqref{eq:flatderivations} the flat system of derivations $\eta_1,\ldots,\eta_\ell$.
After having collected all necessary background material in \Cref{sec:prelimreflectiongroupps,,sec:prelimflat}, we are now able to state and prove our main results.
For $1 \leq j \leq \ell$ and $k \in \BBZ$, define

\begin{equation*}
  \xi_j^{(k)} \eqdef \prim^{k}(\eta_j)
\end{equation*}
and set 
\begin{equation*}
  \Xi^{(k)} \eqdef \big\{ \xi_1^{(k)},\dots,\xi_\ell^{(k)} \big\} \quad \text{ and } \quad \Xi \eqdef \bigcup_{k \in \BBZ}\Xi^{(k)}.
\end{equation*}
Thanks to~\cite[(3.16)]{HMRS2017}, we make the crucial observation that
\begin{equation}
\label{eq:Xi}
  \xi_j^{(1)}  \doteq  \partial_{t_{\ell+1-j}} \quad \text{for}\quad  1 \le j \le \ell.
\end{equation}

\begin{theorem}
\label{thm:Hodgefiltration}
  Let $k \in \BBZ$.
  Then the following hold:
  \begin{enumerate}[(1)]
    \item \label{it:sxik} the $S$-module $\FD(\CA,-k\omega+1)$ is free with basis $\Xi^{(k)}$,
    \item \label{it:rxik} the $R$-module $\FD(\CA,-k\omega+1)^W$ is free with basis $\Xi^{(k)}$,
    \item \label{it:txik} the $T$-module $\FD(\CA,-k\omega+1)^W$ is free with basis $\cup_{p \leq k}\Xi^{(p)}$, and
    \item \label{it:txi} the $T$-module $\FDinfty^W$ is free with basis $\Xi$.
  \end{enumerate}
\end{theorem}

We prove this theorem in \Cref{sec:basisconst}.
It allows us to extend the Hodge filtration of $\DerR$ in~\eqref{eq:hodgefil} to a \defn{Hodge filtration} of the $W$-invariant logarithmic vector fields $\FDinfty^W$ given by the $T$-module
\begin{equation*}
\label{eq:hodgefil2}
  \Hodgefil^{(k)} \eqdef \bigoplus_{-\infty < i \leq k} \Hodgedec_i
\end{equation*}
for $k \in\BBZ$, where $\Hodgedec_i = \primk[-i](\Hodgedec_0)$ is as in~\eqref{eq:hodgedec}.
We record this in the following corollary.

\begin{corollary}
\label{cor:Hodgefiltration}
  We have
  \[
    \FD(\CA,-k\omega+1)^W = \Hodgefil^{(k)} \qquad \text{and}\qquad \FDinfty^W = \bigoplus_{k \in \BBZ} \Hodgefil^{(k)}
  \]
  as $T$-modules and $\prim$ induces $T$-linear isomorphisms
  \begin{align*}
    \prim &: \hspace*{25pt}\Hodgefil^{(k)}\hspace*{25pt} \tilde\longrightarrow \hspace*{25pt}\Hodgefil^{(k+1)}, \\
    \prim &: \FDinfty^W\ \tilde\longrightarrow\ \FDinfty^W.
  \end{align*}
\end{corollary}

\begin{remark}
  In the real case, \Cref{thm:Hodgefiltration} and \Cref{cor:Hodgefiltration} are the main results of~\cite{AT2010}.
  See in particular \cite[Theorems~1.1 and~1.2]{AT2010} for these statements in the notion of logarithmic $1$-forms, and \cite[Theorems~3.7 and~3.9]{AT2010} for the statements in the notion as given here.
  For $k \leq 0$ the first property was also given in the real case in~\cite[Cor.~10]{Yos2002}.
\end{remark}

\begin{remark}
\label{rem:tocheck}
  Denote by $\Omega(\CA,\nu)$ the $S$-dual to $\FD(\CA,\nu)$ for a given multiplicity function~$\nu$ induced by the fixed~$W$-invariant Hermitian form.
  This is, $\Omega(\CA,\nu)$ is the \defn{module of differential $1$-forms} with poles of order $\nu(H)$ along~$H \in \CA$, and $\Omega(\CA,\infty)$ is the \defn{module of logarithmic $1$-forms}.
  If~$W$ is a real reflection group, then it is known that
  \[
    \Omega(\CA,1) \cong \FD(\CA,-\omega+1) = \FD(\CA,-1)
  \]
  as graded $S$-modules, see~\cite[Thm.~2]{AY2009}.
  It turns out that these modules are \emph{not isomorphic} as graded $S$-modules if~$W$ is not real.
  This follows from the observation that both $S$-modules are free, and the polynomial degrees of their homogeneous generators are
  \begin{align*}
    \exp\big(\FD(\CA,-\omega+1)\big) &= \{e_1^*-h,\ldots,e_\ell^*-h\} = \{-e_1,\ldots,-e_\ell\} \\
    \exp\big(\Omega(\CA,1)     \big) &= \hspace*{117.5pt}               \{-e_1^*,\ldots,-e_\ell^*\},
  \end{align*}
  and the fact that 
  \[
    \{e_1,\ldots,e_\ell\} = \{e_1^*,\ldots,e_\ell^*\} \Longleftrightarrow W \text{ is real}.
  \]
\end{remark}

Armed with the universality properties obtained in \Cref{sec:universality1}, we obtain the freeness of $\FD(\CA,-k\omega)$ from \Cref{thm:Hodgefiltration}\eqref{it:sxik} based on the universality of $\primk(E)$ in \Cref{sec:universality2}.

\begin{theorem}
\label{thm:evenfreeness}
  Let $k \in \BBZ$.
  The $S$-module $\FD(\CA,-k\omega)$ is free with basis
  \[
    \big\{ \nabla_{\partial_{x_1}}( \xi_1^{(k)} ),\ldots, \nabla_{\partial_{x_1}}( \xi_\ell^{(k)} ) \big\}.
  \]
\end{theorem}

Observe that this basis is not $W$-invariant, 
in contrast to the bases constructed in \Cref{thm:Hodgefiltration}\eqref{it:sxik} and \eqref{it:rxik}.
Indeed, one cannot expect a $W$-invariant basis of $\FD(\CA,-k\omega)$ as shown in the following theorem which generalizes~\cite[Prop.~5.2]{ATW2012}.

\begin{theorem}
\label{thm:Rindependentbasiscontradiction}
  Let $k \in \BBZ$.
  Then
  \[
    \FD(\CA,-k\omega)^W = \FD(\CA,-k\omega+1)^W.
  \]
  In particular, $\FD(\CA,-k\omega)$ does not have a $W$-invariant basis.
\end{theorem}

\begin{proof}
  Let $\theta \in \FD(\CA,-k\omega)^W$.
  Fix $H \in \CA$ and let $\alpha_H = x_1,x_2,\ldots,x_\ell$ be an orthonormal basis of~$V^*$.
  Write $\theta = \sum_i (f_i / Q_H^n ) \partial_{x_i}$ with $f_1 \in \alpha_H^{-ke_H}\cdot S$ and $f_i \in S$ for $2 \leq i \leq \ell$, as given in~\eqref{eq:nushift}.
  Then $\theta(\alpha_H) = f_1 / Q_H^n = \alpha_H^{-ke_H} g / Q_H^n$ with $g \in S$.

  \medskip

  We aim to show that $g \in \alpha_H \cdot S$.
  To this end let $r \in \refl$ be the reflection along~$H$ and let $\epsilon$ be the $e_H$-th root of unity such that $r(\alpha_H) = \epsilon\alpha_H$.

  Observe first that $r(Q_H) = Q_H$.
  This is because $Q_H = Q / \alpha_H$ so
  \[
    r(Q_H) = r(Q) / r(\alpha_H) = \epsilon Q / \epsilon\alpha_H = Q_H.
  \]
  This yields
  \[
    r\big(\theta(\alpha_H)\big) = \frac{ r(\alpha_H^{-ke_H}) r(g) }{r(Q_H^n)} = \frac{\alpha_H^{-ke_H}}{Q_H^n} r(g).
  \]
  On the other hand,
  \[
    r\big(\theta(\alpha_H)\big) = (r\theta)(r(\alpha_H)) = \epsilon\theta(\alpha_H),
  \]
  since $\theta$ is $W$-invariant and we obtain $r(g) = \epsilon g$.
  With \Cref{lem:quasiinvarinantpoly}, we conclude $g \in \alpha_H\cdot S$ and $\theta \in \FD(\CA,-k\omega+1)^W$.

  \medskip

  The claim that $\FD(\CA,-k\omega)$ does not admit a $W$-invariant basis follows because
  \[
    \FD(\CA,-k\omega+1) \subsetneqq \FD(\CA,-k\omega)\,,
  \]
  by \Cref{lem:compmultiplicities}.
\end{proof}

The proof of \Cref{thm:Rindependentbasiscontradiction} indeed provides the following stronger result.

\begin{corollary}
  Let $\nu$  be any $\BBZ$-valued multiplicity function on $\CA = \CA(W)$. For each $H \in \CA$ choose $a_H \in \BBZ$ so that 
  $\mu : \CA \rightarrow \BBZ$ defined by $\mu(H) := a_H e_H+1$ satisfies $\mu - \omega < \nu \leq \mu$.
  Then
  \[
    \FD(\CA,\nu)^W = \FD(\CA,\mu)^W.
  \]
\end{corollary}

\subsection{Construction of the bases}
\label{sec:basisconst}

We prove \Cref{thm:Hodgefiltration} in several steps.

\begin{proposition}
\label{prop:primwelldef}
  Let $k \in \BBZ$.
  The map
  \[
    \prim : \FD(\CA,-k\omega+1)^W \longrightarrow \FD(\CA,-(k+1)\omega+1)^W
  \]
  is well-defined.
\end{proposition}
\begin{proof}
  In symbols, the well-definedness of $\prim$ means
  \[
    \theta \in \FD(\CA,-k\omega+1)^W \quad \Longrightarrow \quad \prim(\theta) \in \FD(\CA,-(k+1)\omega+1)^W.
  \]
  This is already known for $k < 0$ by~\cite[Thm.~3.22]{HMRS2017}. So we aim to show this for $k \geq 0$.
  Recall that $J = \prod \alpha_H^{ke_H-1}$.
  Let $H \in \CA$ and consider an orthonormal basis $\alpha_H = x_1,x_2,\ldots,x_\ell$ of $V^*$.
  In this basis, for~$\theta$ to belong to $\FD(\CA,-k\omega+1)^W$ entails  that there exist $f_1,\ldots,f_\ell \in S$ such that
  \[
    \theta = \sum_{i=1}^\ell f_iJ^{-1} \partial_{x_i}\quad\text{ and }\quad f_i \in x_1^{ke_H-1} \cdot S \quad \text{for } 2\leq i \leq \ell.
  \]
  Moreover,
  \[
    \prim(\theta) = \sum_{i=1}^\ell D\left(f_iJ^{-1}\right) \partial_{x_i}.
  \]
  Since
  \[
    \prod\alpha_H^{(k+1)e_H-1} D\left(f_1J^{-1}\right) \in S,
  \]
  it remains to show that
  \[
    g_i \eqdef \prod \alpha_H^{(k+1)e_H-1} D\left(f_iJ^{-1}\right) \in x_1^{(k+1)e_H-1} \cdot S
  \]
  for $2 \leq i \leq \ell$.
  First, we see that $g_i \in x_1^{ke_H} \cdot S$, because $f_i \in x_1^{ke_H-1} \cdot S \text{ for } 2\leq i \leq \ell$.
  On the other hand, observe that $g_i$ is  a relative invariant for $\det$, \cf~\cite[Sec.~6.2]{orlikterao:arrangements}.
  This means in particular that
  \[
    s(g_i) = \det(s)(g_i),
  \]
  where~$s = s_H$ is a reflection along~$H$ generating~$W_H$.
  This is because $f_iJ^{-1}$ is $s$-invariant by assumption on~$\theta$,~$D$ is~$W$-invariant, and $J$ is a relative invariant for $\det$, see~\cite[Sec.~6.2]{orlikterao:arrangements}.
  Since $s(\alpha_H) = \det^{-1}(s)\alpha_H$ and $\det(s) = \det^{-(e_H-1)}(s)$, \Cref{lem:quasiinvarinantpoly} implies that $g_i \in x_1^{ke_H} \cdot S$ is indeed contained in $x_1^{ke_H + (e_H-1)} \cdot S$, as desired.
\end{proof}

Recall the diagonal matrix $\Binf$ from~\eqref{eq:Binf} and the Saito matrix $M_\eta$ of the system of flat derivations $\eta_1, \ldots, \eta_\ell$ from \eqref{eq:flatderivations}.

\begin{proposition}
  For $k \geq 1$, we have
  \begin{equation}
    \label{eq:posprimiso}
    \primk\colvect = -M_\eta^{-1}(\Binf + k\one_\ell) \primk[k-1]\colvect.
  \end{equation}
\end{proposition}

\begin{proof}
  We have already seen this to be true for $k=1$ in \Cref{prop:mainingredient}.
  We assume \Cref{eq:posprimiso} to hold for a given~$k$.
  Multiplying both sides of this equation by $-M_\eta$ and applying $\prim$ gives
  \begin{align*}
    (\Binf + k\one_\ell) \primk\colvect
      &= -\prim\left(M_\eta\primk\colvect\right) \\
      &= -\left( \primk\colvect + M_\eta \primk[k+1]\colvect\right)
  \end{align*}
  by the Leibniz rule and the fact that $D(M_\eta) = \one_\ell$ in~\eqref{eq:MV}.
  This yields
  \[
    \primk[k+1]\colvect = -M_\eta^{-1}(\Binf + (k+1)\one_\ell) \primk\colvect,
  \]
  as desired.
\end{proof}

\begin{corollary}
\label{cor:linindep}
  Let $k \in \BBZ$.
  The set
  \[
    \big\{ \primk(\partial_{t_1}), \ldots, \primk(\partial_{t_\ell}) \big\}
  \]
  is linearly independent over~$S$.
\end{corollary}
\begin{proof}
  This is already known for $k < 0$ from~\cite[Prop.~3.18]{HMRS2017} and is trivially true for $k = 0$.
  The case $k > 0$ follows inductively from the previous proposition since $-M_\eta^{-1}(\Binf + k\one_\ell)$ is non-singular for all~$k$.
  Finally, the base change between~$R$ and~$S$ given by $\Jxt$ is non-singular.
\end{proof}

\begin{proof}[Proof of \Cref{thm:Hodgefiltration}~\eqref{it:sxik} and~\eqref{it:rxik}]
  \Cref{prop:primwelldef} and \Cref{cor:linindep} allow us to apply our version of Saito's criterion, \Cref{thm:zieglersaito}.
  For this, it only remains to check that $\sum_i \pdeg(\primk(\eta_i)) = |-k\omega+1|$.
  This follows from 
  \begin{align*}
    \sum_{i=1}^\ell \pdeg(\primk(\eta_i))
      &= \sum_{i=1}^\ell \big(k\pdeg(D) - k + \pdeg(\eta_i) \big) \\
      &= \sum_{i=1}^\ell \big(-kh + e^*_i \big) \\
      &= -k\ell h + |\CA| = |-k\omega+1|,
  \end{align*}
  by~\eqref{eq:coxeter},~\eqref{eq:expcoexp}, and~\cite[Thm.~4.23]{orlikterao:arrangements}.
\end{proof}

In order to prove \Cref{thm:Hodgefiltration}~\eqref{it:txik} and~\eqref{it:txi}, it remains to show that the set given by $\cup_{p \leq k}\Xi^{(p)}$ is $T$-independent and generates $\FD(\CA,-k\omega+1)^W$ as a $T$-module.

\begin{proposition}
\label{prop:Tindependent}
  The set $\Xi = \big\{ \xi_j^{(k)} \mid k \in \BBZ, 1\leq j \leq \ell \big\}$ is $T$-independent.
\end{proposition}

\begin{proof}
  Suppose $\Xi$ is not $T$-independent.
  This means that there exist $d \leq e$ and $a_{kj} \in T$ with $a_{dj}, a_{ek}$ not equal to zero for some $1 \leq j,k \leq \ell$ so that
  \begin{equation}
  \label{eq:Tcontradiction}
    \sum_{k=d}^e \sum_{i=1}^\ell a_{ki} \xi_i^{(k)} = \sum_{k=d}^e (a_{k1},\ldots,a_{k\ell})(\xi_1^{(k)},\ldots,\xi_\ell^{(k)})^\tr = 0.
  \end{equation}
  It follows from~\eqref{eq:posprimiso} that there is a non-singular matrix~$N_e$ so that
  \[
    \pmat{\xi_1^{(e)} \\ \vdots \\ \xi_\ell^{(e)}} = \primk[e] \colvect = N_e \colvect.
  \]
  Moreover, setting $B_{k+1} = -(\Binf + (k+1)\one_\ell)$, we obtain
  \[
    B_{k+1}^{-1}M_\eta\primk[k+1]\colvect = \primk\colvect.
  \]
  Recall from ~\eqref{eq:flatderivations} that $D[M_\eta]=\one_\ell$ and $D[B_k]=0$ for all~$k \in \BBZ$, where the application of~$D$ to a matrix means its application to every entry of the matrix.
  Thus, the non-singular matrix $H_k = B_{k+1}^{-1}M_\eta\cdots B_{e-1}^{-1}M_\eta B_e^{-1}M_\eta$ for $d \leq k \leq e$ (with the convention that $H_e = \one_\ell$) yields
  \[
    \pmat{ \xi_1^{(k)} \\ \vdots \\ \xi_\ell^{(k)}} = H_kN_e\colvect.
  \]
  Given this description,~\eqref{eq:Tcontradiction} is equivalent to 
  \[
  \sum_{k=d}^e (a_{k1},\ldots,a_{k\ell}) H_k = {\bf 0},
  \]
  where we used that~$N_e$ is non-singular.
  Apply $D^{e-d}$ and use that $D^{e-k+1}(H_k)={\bf 0}$ to obtain
  \[
  (a_{e1},\ldots,a_{e\ell}) B_{d+1}^{-1} D[M_\eta] B_{d+2}^{-1}D[M_\eta] \cdots B_e^{-1}D[M_\eta] = {\bf 0}.
  \]
  Since $D[M_\eta]=\one_\ell$, this implies $(a_{e1},\ldots,a_{e\ell}) = {\bf 0}$, a contradiction.
\end{proof}

\begin{proposition}
\label{prop:Tgenerating}
  With the notation as above, viewed as $T$-modules,
  \begin{equation}
  \label{eq:Tdecomp}
    \FD(\CA,-k\omega+1)^W=\bigoplus_{p \le k} \bigoplus_{i=1}^\ell T \cdot \xi_i^{(p)}.
  \end{equation}
\end{proposition}
\begin{proof}
  We know that the right hand side of~\eqref{eq:Tdecomp} is a direct $T$-module summand of the space on the left by \Cref{prop:Tindependent}.
  It thus remains to show that the right hand side generates the left hand side.
  Owing to \Cref{thm:Hodgefiltration}\eqref{it:rxik}, we know that $\xi_1^{(k)},\ldots,\xi_\ell^{(k)}$ form an $R$-basis for $\FD(\CA,-kw+1)^W$.
  Thus it suffices to show that the right hand side is an $R$-module.
  We know this is true if $k \le 0$ by~\cite[Thm.~3.22]{HMRS2017} and proceed by induction on~$k \geq 0$.
  From
  \[
    \prim\big(t_\ell\cdot(\xi_1^{(k)},\ldots,\xi_\ell^{(k)})\big) =
      D(t_\ell)\cdot(\xi_1^{(k)},\ldots,\xi_\ell^{(k)}) + t_\ell\cdot(\xi_1^{(k+1)},\ldots,\xi_\ell^{(k+1)}),
  \]
  we obtain
  \[
    t_\ell\cdot (\xi_1^{(k+1)},\ldots,\xi_\ell^{(k+1)}) = \prim\big(t_\ell\cdot(\xi_1^{(k)},\ldots,\xi_\ell^{(k)})\big) -
      (\xi_1^{(k)},\ldots,\xi_\ell^{(k)}).
  \]
  By induction hypothesis, we have
  \[
    t_\ell\xi_j^{(k)} \in \bigoplus_{p \le k}\bigoplus_{i=1}^\ell T \cdot \xi_i^{(p)}
  \]
  for $1 \leq j \leq \ell$, and thus
  \[
    \prim(t_\ell \xi_j^{(k)}) \in \prim\left( \bigoplus_{p \le k}\bigoplus_{i=1}^\ell T \cdot \xi_i^{(p)}\right)
      = \bigoplus_{p \le k}\bigoplus_{i=1}^\ell T \cdot \prim(\xi_i^{(p)})
      = \bigoplus_{p \le k+1}\bigoplus_{i=1}^\ell T \cdot \xi_i^{(p)}.
  \]
  We deduce that for every $1 \leq j \leq \ell$, we get
  \[
    t_\ell \xi_j^{(k+1)} \in \bigoplus_{p \le k+1} \bigoplus_{i=1}^\ell T \cdot \xi_i^{(p)}
  \]
  to complete the proof.
\end{proof}

\begin{proof}[Proof of \Cref{thm:Hodgefiltration}~\eqref{it:txik} and~\eqref{it:txi}]
    \Cref{prop:Tindependent,,prop:Tgenerating} immediately give statement~\eqref{it:txik}.
  Finally,~\eqref{it:txi} follows, because every element in $\FD(\CA,-\infty)^W$ is contained in $\FD(\CA,-k\omega+1)^W$ for some $k \in \BBZ$.
\end{proof}

\subsection{Universal vector fields for reflection arrangements}
\label{sec:universality2}

In this final section, we conclude the universality properties for reflection arrangements.

\medskip

Our first result generalizes~\cite[Prop.~2.5]{Wak2011} to the complex case, and in particular, it implies \Cref{thm:evenfreeness}.
Recall the Euler derivation~$E$ from~\eqref{eq:Euler}.

\begin{theorem}
\label{thm:universality}
  With the notation as above, $\primk(E)$ is $(-k\omega)$-universal for any $k \in \BBZ$.
\end{theorem}

\begin{proof}
  For $k \leq 0$, this is the special case of~\cite[Thm.~3.20]{HMRS2017} for $\theta_i = \partial_{x_i}$.
  It remains to consider the case $k \geq 0$.
  Set $E_k \eqdef \primk(E)$.
  By \Cref{prop:primwelldef}, we have that $E_k \in \FD(\CA,-k\omega+1)$ which in turn implies that $\Phi_{E_k} : \DerS \rightarrow \FD(\CA,-k\omega)$ is well-defined, by \Cref{lem:booleanmultiplicitywelldefined}.
  Moreover, the $S$-independence of $\{ \nabla_{\partial_{x_i}}(E_k) \mid 1 \leq i \leq \ell \}$ follows from the $R$-independence of $\{ \nabla_{\partial_{t_i}}(E_k) \} = \{ \tfrac{1}{h} \nabla_D^k(\partial_{t_i}) \}$ proven in \Cref{thm:Hodgefiltration}\eqref{it:rxik} and~\eqref{eq:Xi}.
  Here, we used that $\nabla_{\partial_{t_i}}\nabla_{\partial_{t_j}} = \nabla_{\partial_{t_j}}\nabla_{\partial_{t_i}}$ and that $\nabla_{\partial_{t_i}}(E) = \tfrac{1}{h}\partial_{t_i}$.
  Finally, we deduce that $\{ \nabla_{\partial_{x_i}}(E_k) \mid 1 \leq i \leq \ell \}$ is an $S$-basis of $\FD(\CA,-k\omega)$ from the version of the Saito criterion given in \Cref{thm:zieglersaito}\eqref{eq:zieglersaito3}.
  In the later, we remark that indeed $\nabla_{\partial_{x_i}}(E_k)$ is homogeneous, since~$E_k$ is homogeneous (\cf~\eqref{eq:pdegnabla}), and so we obtain
  \begin{equation*}
    \sum \pdeg\big(\nabla_{\partial_{x_i}}(E_k)\big)
      = \sum \big(-kh + \pdeg(\partial_{x_i})\big)
      = -k\ell h = |-k\omega|,
  \end{equation*}
  again by~\eqref{eq:coxeter},~\eqref{eq:expcoexp}, and~\cite[Thm.~4.23]{orlikterao:arrangements}.
\end{proof}

The following properties are the complex counterparts of the results in~\cite[Sec.~2]{Wak2011}, where we remark that the arguments are similar.
Recall the order multiplicity $\omega : \CA \rightarrow \BBZ$ on the reflection arrangement~$\CA$ given by $\omega(H) = e_H$.
Our aim here is to generalize \Cref{thm:01multi} to multiplicity functions $\mu : \CA \rightarrow \BBZ$ with $-\omega+1 \le \mu \le 1$.

\begin{proposition}
\label{prop:-1multi}
  Let $\nu$ be any $\BBZ$-valued multiplicity on~$\CA$ and let $\zeta \in \FDinfty^W$ be $\nu$-universal.
  Then
  \[
    \Phi_\zeta : \FD(\CA,-\omega+1) \rightarrow \FD(\CA,\nu-\omega+1)
  \]
  is an isomorphism of~$S$-modules.
\end{proposition}

\begin{proof}
  We have seen in \Cref{thm:Hodgefiltration}\eqref{it:sxik} that $\{ \partial_{t_1},\ldots,\partial_{t_\ell} \}$ is an $S$-basis of $\FD(\CA,-\omega+1)$.
  We thus aim to show that $\{ \nabla_{\partial_{t_1}}\zeta,\ldots,\nabla_{\partial_{t_\ell}}\zeta\}$ is an $S$-basis of $\FD(\CA,\nu-\omega+1)$.

  It is immediate from~\eqref{eq:defD} that $J \partial_{t_j} \in \DerS$.
  This implies $J \nabla_{\partial_{t_j}}\zeta \in \FD(\CA,\nu)$ by the linearity of~$\nabla$ in the first parameter and the $\nu$-universality of~$\zeta$.
  Thus we have that $J(\nabla_{\partial_{t_j}}\zeta)(\alpha_H) \in \alpha_H^{\nu(H)}\SalphaH$ or, equivalently,
  \[
    (\nabla_{\partial_{t_j}}\zeta)(\alpha_H) \in \alpha_H^{\nu(H)-(e_H-1)}\SalphaH
  \]
  for a given $H \in \CA$.
  Now let $\alpha_H=x_1,x_2,\ldots,x_\ell$ be an orthonormal basis of~$V^*$ and for $2 \leq i \leq \ell$ set
  \[
    g_i = Q_H^N J(\nabla_{\partial_{t_j}}\zeta)(x_i) \in S,
  \]
  for some $N \geq 0$, where $Q_H = Q/\alpha_H$.
  Analogous to the argument in the proof of \Cref{prop:primwelldef}, we obtain that $g_i \in \alpha_H^{e_H-1}S$, and so
  \[
    (\nabla_{\partial_{t_j}}\zeta)(x_i) = Q_H^{-N} g_i J^{-1} \in \SalphaH,
  \]
  for $2 \leq i \leq \ell$.
  We thus have $\nabla_{\partial_{t_j}}\zeta \in \FD(\CA,\nu-\omega+1)$.
  Computing
  \[
    \det\big( (\nabla_{\partial_{t_j}}\zeta)(x_i) \big) \doteq J^{-1} \det\big((\nabla_{\partial_{x_j}}\zeta)(x_i) \big) \doteq J^{-1} Q_{\nu} = Q_{\nu-\omega+1}
  \]
  yields the statement, thanks to \Cref{thm:zieglersaito}.
\end{proof}

\begin{theorem}
\label{thm:zetapmshifts}
  Let $\zeta \in \FDinfty^W$ be $\nu$-universal for a $\BBZ$-valued multiplicity function~$\nu$ on~$\CA$.
  Let $\mu : \CA \rightarrow \BBZ$ be a multiplicity with $-\omega+1 \leq \mu \leq 1$.
  Then
  \[
    \Phi_\zeta : \FD(\CA,\mu) \rightarrow \FD(\CA,\nu+\mu)
  \]
  is an isomorphism of~$S$-modules.
  In particular, $\FD(\CA,\mu)$ is free if and only if $\FD(\CA,\nu+\mu)$ is free.
\end{theorem}

Before proving \Cref{thm:zetapmshifts}, we use the fact that $\zeta = \primk[-m](E)$ is $(m\omega)$-universal for any $m\in \BBZ$, by~\Cref{thm:universality} to obtain the special case of the \defn{generalized Ziegler multiplicity} on a reflection arrangement~$\CA$ given by $m\omega-1$ for $m \in \BBZ$.
For $m=1$, this multiplicity was originally considered by Ziegler~\cite{ziegler:multiarrangements} (assigning the number of reflections along~$H$ to any reflecting hyperplane), and shown  to be free in general in~\cite{ziegler:multiarrangements} and~\cite{HR2014}.

\begin{corollary}
  Let $\{ e_1,\ldots,e_\ell\}$ be the set of exponents of $(\CA,\omega-1)$.
  For any $m \in \BBZ$, the multi-arrangement $(\CA, m\omega-1)$ is free and its exponents are given by
  \[
    \big\{ e_1+(m-1)h,\dots,e_\ell+(m-1)h \big\}.
  \]
\end{corollary}

\begin{proof}[Proof of \Cref{thm:zetapmshifts}]
  As $\FD(\CA,\mu) \subseteq \FD(\CA,-\omega+1)$ and $\FD(\CA,\nu+\mu) \subseteq \FD(\CA,\nu-\omega+1)$, \Cref{prop:-1multi} implies that $\Phi_\zeta$ is injective on $\FD(\CA,\mu)$.
  It thus remains to show that
  \begin{equation}
  \label{eq:zetampshifts}
    \Phi_\zeta\big(\FD(\CA,\mu)\big) = \FD(\CA,\nu+\mu).
  \end{equation}
  Fix $H \in \CA$ and let $\alpha_H=x_1,x_2,\ldots,x_\ell$ be an orthonormal basis of~$V^*$.
  The element $g = \alpha_H^{-\nu(H)-1}\zeta(\alpha_H) \in \SalphaH$ is a unit in $\SalphaH$, by \Cref{lem:universalunit}.
  For $\theta \in \FDinfty$, compute
  \begin{align*}
    \big(\Phi_\zeta(\theta)\big)(\alpha_H)
      &= \theta(\alpha_H^{\nu(H)+1}g) \\
      &= \alpha_H^{\nu(H)+1}\theta(g) + (\nu(H)+1)\alpha_H^{\nu(H)}\theta(\alpha_H) g \\
      &= \alpha_H^{\nu(H)}\theta(\alpha_H) \big( \alpha_H \tfrac{\partial g}{\partial \alpha_H} + (\nu(H)+1)g\big) + \alpha_H^{\nu(H)+1}\sum_{i=2}^\ell \theta(x_i) \tfrac{\partial g}{\partial x_i} \\
      &= \alpha_H^{\nu(H)} \theta(\alpha_H) U + \alpha_H^{\nu(H)+1} C,
  \end{align*}
  where $U = \alpha_H \frac{\partial g}{\partial \alpha_H} + (\nu(H)+1)g$ is again a unit in $\SalphaH$ and $C = \sum_{i=2}^\ell \theta(x_i) \tfrac{\partial g}{\partial x_i}$.
  Note that $C \in \SalphaH$, since $\theta \in \FDinfty$, and so is $\alpha_H^{1-\mu(H)}C$.
  We thus obtain
  \[
    \alpha_H^{-\nu(H)-\mu(H)}\big(\Phi_\zeta(\theta)\big)(\alpha_H) = \alpha_H^{-\mu(H)}\theta(\alpha_H) U + \alpha_H^{1-\mu(H)} C,
  \]
  implying that
  \[
    \big(\Phi_\zeta(\theta)\big)(\alpha_H) \in \alpha_H^{\nu(H)+\mu(H)} \SalphaH \Leftrightarrow \theta(\alpha_H) \in \alpha_H^{\mu(H)} \SalphaH.
  \]
  Since this holds for any $H \in \CA$, the reverse implication shows the containment ``$\subseteq$'' in~\eqref{eq:zetampshifts} and the forward implication the containment ``$\supseteq$''.
\end{proof}

We finally conclude the extension of~\cite[Thm.~2.8]{Wak2011} to the unitary setting.

\begin{theorem}
\label{thm:shifteduniversality}
  Let $\zeta \in \FDinfty$ and let $k \in \BBZ$.
  Then
  \[
    \zeta \text{ is } (k\omega)\text{-universal } \Longrightarrow \priminv\zeta\text{ is } \big((k+1)\omega\big)\text{-universal}.
  \]
\end{theorem}

\begin{proof}
  We set $\nu = k\omega$ in the proof for better readability.
  Let $\eta_1,\ldots,\eta_\ell \in \FD(\CA,1)^W$ be an $S$-basis.
  Due to the universality of $\zeta$, we also have that $\nabla_{\eta_1}\zeta,\ldots,\nabla_{\eta_\ell}\zeta \in \FD(\CA,\nu+1)^W$ form an $S$-basis.
  Since~$\nu$ is a multiple of~$\omega$, we also have that
  \[
    \prim\nabla_{\eta_j}\zeta \in \FD(\CA,\nu+1-\omega)^W.
  \]
  Because the $\partial_{t_j} \in \FD(\CA,-\omega+1)^W$ form an $S$-basis as well, we obtain
  \[
    \prim\nabla_{\eta_j}\zeta = \sum_i f_{ij} \nabla_{\partial_{t_i}}\zeta
  \]
  for $W$-invariant polynomials~$f_{ij}$.
  Comparing polynomial degrees yields $\deg f_{ij} = e_i + e_j^* - h < h$ (or, respectively, $f_{ij} = 0$ for $e_i + e_j^* < h$) which shows that $f_{ij} \in T$.
  Moreover, $\deg f_{i,\ell+1-i} = e_i + e_{\ell+1-i}^* - h = 0$ for all $1 \leq i \leq \ell$, and it follows that $\det f_{ij} \in \BBC$.
  Applying $\priminv$ gives
  \[
    \nabla_{\eta_j}\zeta = \sum_i f_{ij} \nabla_{\partial_{t_i}}\priminv\zeta.
  \]
  Because the vector fields on the left form an $S$-basis of $\FD(\CA,\nu+1)$, we readily see that $\det f_{ij} \neq 0$ and the $\nabla_{\partial_{t_i}}\priminv\zeta \in \FD(\CA,\nu+1)^W$ form an $S$-basis.
  Again, because~$\nu$ is a multiple of~$\omega$, we have $\priminv\zeta \in \FD(\CA,\nu+1-\omega)^W$ and we see that the derivations
  \[
    \nabla_{\partial_{x_j}}\priminv\zeta = \sum_i \Jtx \nabla_{\partial_{t_j}} \priminv \zeta
  \]
  form an $S$-basis of $\FD(\CA,\nu+\omega)$ by Saito's criterion.
\end{proof}

\subsection*{Acknowledgments}

We thank the referees for their helpful comments clarifying some points. 

\subsection*{Funding}

The research of this work was supported by the DFG (Grant RO 1072/16-1 to G.~R\"ohrle and Grants STU 563/2 and STU 563/4-1 to C.~Stump) and by JSPS KAKENHI (JP16H03924 to T.~Abe and JP15KK0144, JP18H01115 to M.~Yoshinaga). 


\bigskip

\bibliographystyle{amsalpha}

\newcommand{\etalchar}[1]{$^{#1}$}
\providecommand{\bysame}{\leavevmode\hbox to3em{\hrulefill}\thinspace}
\providecommand{\MR}{\relax\ifhmode\unskip\space\fi MR }
\providecommand{\MRhref}[2]{%
  \href{http://www.ams.org/mathscinet-getitem?mr=#1}{#2} }
\providecommand{\href}[2]{#2}


\end{document}